\documentclass{article}
\usepackage{amsmath}
\usepackage{amsthm}
\usepackage{amsfonts}
\usepackage{amsfonts}

\usepackage[all]{xy}

\def\Br{\operatorname{Br}}

\newtheorem{theorem}{Theorem}
\newtheorem{lemma}[theorem]{Lemma}

\begin{document}

\title{Branch points and stability}
\author{J.F. Jardine}

\maketitle

\begin{abstract}
The hierarchy poset and  branch point poset for a data set both admit a calculus of least upper bounds. A method involving upper bounds is used to show that the map of branch points associated to the inclusion of data sets is a controlled homotopy equivalence, where the control is expressed by an upper bound relation that is constrained by Hausdorff distance.
  \end{abstract}

\nocite{persist-htpy}

\bigskip



\section*{Introduction}

This paper is a discussion of
 clustering phenomena that arise in connection with inclusions $X \subset Y \subset \mathbb{R}^{n}$ of data sets, interpreted through the lens of hierarchies of clusters and branch points.
\medskip

Suppose that $X$ is a finite subset (a data set) in a metric space $Z$. There is a well known system of simplicial complexes $V_{s}(X)$ whose simplices are the subsets $\sigma$ of $X$ such that $d(x,y) \leq s$ for each pair of points $x,y \in \sigma$, where $d$ is the metric on $Z$. The complexes $V_{s}(X)$ are the Vietoris-Rips complexes for the data set $X$.

If $k$ is a positive integer, $L_{s,k}(X)$ is the subcomplex of $V_{s}(X)$ whose simplices $\sigma$ have vertices $x$ such that $d(x,y) \leq s$ for at least $k$ distinct points $y \ne x$ in $X$. This object is variously called a degree Rips complex, or a Lesnick complex. The number $k$ is a density parameter.

The simplicial complexes $V_{s}(X)$ and $L_{s,k}(X)$ are defined by their respective partially ordered sets (posets) of simplices $P_{s}(X)$ and $P_{s,k}(X)$ \cite{persist-htpy}. The corresponding nerves $BP_{s}(X)$ and $BP_{s,k}(X)$ are barycentric subdivisions of the respective complexes $V_{s}(X)$ and $L_{s,k}(X)$, and therefore have the same homotopy types. This identification of homotopy types is assumed in this paper, so that $V_{s}(X) = BP_{s}(X)$ and $L_{s,k}(X) = BP_{s,k}(X)$, respectively.

A relationship $s \leq t$ between spatial parameters induces an inclusion
\begin{equation*}
  L_{s,k}(X) \subset L_{t,k}(X).
\end{equation*}
Some of the complexes $L_{s,k}(X)$ could be empty, and $L_{s,k}(X)$ is the barycentric subdivision of a big simplex for $s$ sufficiently large if $k$ is bounded above by the the cardinality of $X$. Observe also that $L_{s,0}(X) = V_{s}(X)$, and that the subobjects $L_{s,k}(X)$ filter $V_{s}(X)$.
\medskip

For a fixed integer $k$, the sets $\pi_{0}L_{s,k}(X)$ of path components, as $s$ varies, define a tree $\Gamma_{k}(X)$ with elements $(s,[x])$ such that $[x] \in \pi_{0}L_{s,k}(X)$. 

The tree $\Gamma_{k}(X)$ is the object studied by the HDBSCAN clustering algorithm, while the individual sets of clusters $\pi_{0}L_{s,k}(X)$ are computed for the DBSCAN algorithm.

The tree $\Gamma_{k}(X)$ has a subobject $\Br_{k}(X)$ whose elements are the branch points of the tree $\Gamma_{k}(X)$.
The branch points of $\Gamma_{k}(X)$ are in one to one correspondence with the stable components for $\Gamma_{k}(X)$ that are defined in \cite{clusters}, in the sense that every stable component starts at a unique branch point. We replace the stable component discussion of \cite{clusters} with the branch point tree $\Br_{k}(X)$, and make particular use of its ordering.

The branch point tree $\Br_{k}(X)$ is a highly compressed version of the hierarchy $\Gamma_{k}(X)$ that is produced by the HDBSCAN algorithm.

We derive a stability result (Theorem \ref{th 2}) for the branch point tree. This result follows from a stability theorem for the degree Rips complex \cite{persist-htpy}, together with a calculus of least upper bounds for the branch point tree that is developed in the next section.
\medskip

Suppose that $i: X \subset Y$ are data sets in $Z$, and that $r > 0$.
Suppose that the Hausdorff distance
$d_{H}(X_{dis}^{k+1},Y_{dis}^{k+1}) < r$ in $Z^{k+1}$, 
where $X^{k+1}_{dis}$ is the set of $k+1$ distinct points in $X$, interpreted as a subset of the product metric space $Z^{k+1}$. The inclusion $i$ induces an inclusion $i: L_{s,k}(X) \to L_{s,k}(Y)$ of simplicial complexes, which is natural in all $s$ and $k$.

The stability theorem for the degree Rips complex (Theorem 6 of \cite{persist-htpy}, which is a statement about posets) implies the following:

\begin{theorem}\label{th 1}
  Suppose that $X \subset Y \subset Z$ are data sets, and we have the relation
\begin{equation*}
  d_{H}(X_{dis}^{k+1},Y_{dis}^{k+1}) < r
\end{equation*}
on Hausdorff distance between associated configuration spaces in $Z^{k+1}$.
Then there is a diagram of simplicial complex maps
\begin{equation}\label{eq 1}
  \xymatrix{
    L_{s,k}(X) \ar[r]^{\sigma} \ar[d]_{i} & L_{s+2r}(X) \ar[d]^{i} \\
    L_{s,k}(Y) \ar[r]_{\sigma} \ar[ur]^{\theta} & L_{s+2r}(Y)
  }
\end{equation}
in which the horizontal and vertical maps are natural inclusions. The upper triangle of the diagram commutes, and the lower triangle commutes up to a homotopy which fixes $L_{s,k}(X)$.
\end{theorem}

Theorem \ref{th 1} specializes to the Rips Stability Theorem in the case $k=0$  (see \cite{persist-htpy}, \cite{BlumLes}).
The picture (\ref{eq 1}) is often called a homotopy interleaving.
\medskip

Application of the path component functor $\pi_{0}$ to the diagram (\ref{eq 1}) gives a commutative diagram
\begin{equation}\label{eq 2}
  \xymatrix{
    \pi_{0}L_{s,k}(X) \ar[r]^{\sigma} \ar[d]_{i} & \pi_{0}L_{s+2r}(X) \ar[d]^{i} \\
    \pi_{0}L_{s,k}(Y) \ar[r]_{\sigma} \ar[ur]^{\theta} & \pi_{0}L_{s+2r}(Y)
  }
\end{equation}
which is an interleaving of clusters. This is true for all homotopy invariants: in particular, application of homology functors to (\ref{eq 1}) produces interleaving diagram in homology groups. 
\medskip

The tree $\Gamma_{k}(X)$ has least upper bounds, and these restrict to least upper bounds for the subtree $\Br_{k}(X)$ of branch points (Lemma \ref{lem 3}).

The inclusion $\Br_{k}(X) \subset \Gamma_{k}(X)$ is a homotopy equivalence of posets, where the homotopy inverse is defined by taking the maximal branch point $(s_{0},[x_{0}]) \leq (s,[x])$ below $(s,[x])$ for each object of $\Gamma_{k}(X)$. The existence of the maximal branch point below an object $(s,[x])$ is a consequence of Lemma \ref{lem 6}.

The poset map $i:\Gamma_{k}(X) \to \Gamma_{k}(Y)$ defines a poset map $i_{\ast}: \Br_{k}(X) \to \Br_{k}(Y)$, via the homotopy equivalences for the data sets $X$ and $Y$ of the last paragraph.
The maps $\theta: \pi_{0}L_{s,k}(Y) \to \pi_{0}L_{s+2r}(X)$ induce morphisms of trees $\theta_{\ast}: \Gamma_{k}(Y) \to \Gamma_{k}(X)$ and $\theta_{\ast}: \Br_{k}(Y) \to \Br_{k}(X)$.

We then have the following:

\begin{theorem}\label{th 2}
Under the assumptions of Theorem \ref{th 1}, there is a homotopy commutative diagram
\begin{equation}\label{eq 3}
  \xymatrix{
    \Br_{k}(X) \ar[r]^{\sigma_{\ast}} \ar[d]_{i_{\ast}} & Br_{k}(X) \ar[d]^{i_{\ast}} \\
    \Br_{k}(Y) \ar[r]_{\sigma_{\ast}} \ar[ur]^{\theta_{\ast}} & \Br_{k}(Y)
  }
\end{equation}
of morphisms of trees. 
\end{theorem}
This paper is devoted to a proof and interpretation of this result.

\section{Branch points and upper bounds}

Fix the density number $k$ and suppose that $L_{s,k}(X) \ne \emptyset$ for $s$ sufficiently large. Apply the path component functor to the $L_{s,k}(X)$, to get a diagram of functions
\begin{equation*}
  \dots \to \pi_{0}L_{s,k}(X) \to \pi_{0}L_{t,k}(X) \to \dots
\end{equation*}

The graph $\Gamma_{k}(X)$ has vertices $(s,[x])$ with $[x] \in \pi_{0}L_{s,k}(X)$, and edges $(s,[x]) \to (t,[x])$ with $s \leq t$. This graph underlies a poset with a terminal object, and is therefore a tree (or hierarchy).

The morphisms of $\Gamma_{k}(X)$ are relations $(s,[x]) \leq (t,[y])$. The existence of such a relation means that  $[x] = [y] \in \pi_{0}L_{t,k}(X)$, or that the image of $[x] \in \pi_{0}L_{s,k}(X)$ is $[y]$ under the induced function $\pi_{0}L_{s,k}(X) \to \pi_{0}L_{t,k}(X)$.
\medskip

\noindent
    {\bf Remarks}:\ 1)\ Partitions of $X$ given by the set $\pi_{0}V_{s}(X)$ are standard clusters. The tree $\Gamma_{0}(X)=\Gamma(V_{\ast}(X))$ defines a hierarchical clustering that is similar to the single linkage clustering.
    \medskip

    \noindent
    2)\ The set $\pi_{0}L_{s,k}(X)$ gives a partitioning of the set of elements of $X$ having at least $k$ neighbours of distance $\leq s$, which is the subject of the DBSCAN algorithm. The tree $\Gamma_{k}(X)=\Gamma(\pi_{0}L_{\ast,k}(X))$ is the structural object underlying the HDBSCAN algorithm.
    \medskip

    A {\it branch point} in the tree $\Gamma_{k}(X)$ is a vertex $(t,[x])$ such that either of following two conditions hold:
    \begin{itemize}
    \item[1)] there is an $s_{0} < t$ such that for all $s_{0} \leq s < t$ there are distinct vertices $(s,[x_{0}])$ and $(s,[x_{1}])$ with $(s,[x_{0}]) \leq (t,[x])$ and $(s,[x_{1}]) \leq (t,[x])$, or
    \item[2)] there is no relation $(s,[y]) \leq (t,[x])$ with $s < t$.
\end{itemize}
The second condition means that a representing vertex $x$ of the path component $[x] \in \pi_{0}L_{t,k}(X)$ is not a vertex of $L_{s,k}(X)$ for $s < t$.    
Write $\Br_{k}(X)$ for the set of branch points $(s,[x])$ in $\Gamma_{k}(X)$.
\medskip

The set $\Br_{k}(X)$ inherits a partial ordering from the poset $\Gamma_{k}(X)$, and the inclusion $\Br_{k}(X) \subset \Gamma_{k}(X)$ of the set of branch points defines a monomorphism of posets. 

Every branch point $(s,[x])$ of $\Gamma_{k}(X)$ has $s=s_{i}$, where $s_{i}$ is a phase change number for $X$. The phase change numbers are the various distances $d(x,y)$ between the elements of the finite set $X$.

The branch point poset $\Br_{k}(X)$ is a tree, because the element $(s,[x])$ corresponding to the largest phase change number $s$ is terminal.
\medskip

Suppose that $(s,[x])$ and $(t,[y])$ are vertices of the graph $\Gamma_{k}(X)$. There is a vertex $(v,[w])$ such that $(s,[x]) \leq (v,[w])$ and $(t,[y]) \leq (v,[w])$. The two relations specify that $[x]=[z]=[y]$ in $\pi_{0}L_{v,k}(X)$.

There is a unique smallest vertex $(u,[z])$ which is an upper bound for both $(s,[x])$ and $(t,[y])$. The number $u$ is the smallest parameter (necessarily a phase change number) such that $[x]=[y]$ in $\pi_{0}L_{u,k}(X)$, and so $[z]=[x]=[y]$.
In this case, 
one writes
\begin{equation*}
  (s,[x]) \cup (t,[y]) = (u,[z]).
  \end{equation*}
The vertex $(u,[z])$ is the {\it least upper bound} (or join) of $(s,[x])$ and $(t,[y])$.

Every finite collection of points $(s_{1},[x_{1}]), \dots ,(s_{p},[x_{p}])$ has a least upper bound
\begin{equation*}
  (s_{1},[x_{1}]) \cup \dots \cup (s_{p},[x_{p}])
\end{equation*}
in the tree $\Gamma_{k}(X)$.

\begin{lemma}\label{lem 3}
  The least upper bound $(u,[z])$ of branch points $(s,[x])$ and $(t,[y])$ is a branch point.
\end{lemma}

\begin{proof}
If there is a number $v$ such that $s,t<v<u$, then $(v,[x])$ and $(v,[y])$ are distinct because $(u,[z])$ is a least upper bound, so that $(u,[z])$ is a branch point.

  Otherwise, $s=u$ or $t=u$, in which case $(u,[z]) = (s,[x])$ or $(u,[z])=(t,[y])$. In either case, $(u,[z])$ is a branch point.
  \end{proof}

It follows from Lemma \ref{lem 3} that any two branch points $(s,[x])$ and $(t,[y])$ have a least upper bound in $\Br_{k}(X)$, and that the poset inclusion $\alpha: \Br_{k}(X) \to \Gamma_{k}(X)$ preserves least upper bounds.
\medskip
  
  We have the following observation:

  \begin{lemma}\label{lem 4}
    Suppose that $(s_{1},[x_{1}]), (s_{2},[x_{2}])$ and $(s_{3},[x_{3}])$ are vertices of $\Gamma_{k}(X)$. Then
    \begin{equation*}
      (s_{1},[x_{1}]) \cup (s_{3},[x_{3}]) \leq ((s_{1},[x_{1}]) \cup (s_{2},[x_{2}])) \cup ((s_{2},[x_{2}]) \cup (s_{3},[x_{3}])).
      \end{equation*}
  \end{lemma}
  \medskip
  
  \noindent
  {\bf Remark}:\
   Carlsson and M\'emoli \cite{CM-2010B} define an ultrametric $d$ on $X=V_{0}(X)$, for which they say that $d(x,y)=s$, where $s$ is the minimum parameter value such that $[x]=[y] \in
 \pi_{0}V_{s}(X)$.

 The least upper bound concept is both an extension of and a potential replacement for this ultrametric, and Lemma \ref{lem 4} is the analog for the triangle inequality.

 The Carlsson-M\'emoli theory does not apply to the full tree $\Gamma_{k}(X)$, because the vertex sets of the Lesnick complexes $L_{s,k}(X)$ can vary with changes of the distance parameter $s$. We can, however, define an ultrametric on each of the sets $\pi_{0}L_{s,k}(X)$ as follows:
\smallskip

 Suppose given $[x]$ and $[y]$ in $\pi_{0}L_{s,k}(X)$ (or equivalently, points $(s,[x])$ and $(s,[y])$ in $\Gamma_{k}(X)$). Write $d([x],[y]) = u-s$, where $(s,[x]) \cup (s,[y]) = (u,[w])$.

\begin{lemma}\label{lem 5}
Every vertex $(s,[x])$ of $\Gamma_{k}(X)$ has a unique largest branch point $(s_{0},[x_{0}])$ such that $(s_{0},[x_{0}]) \leq (s,[x])$.
\end{lemma}

\begin{proof}
The least upper bound of the finite list of the branch points $(t,[y])$ such that $(t,[y]) \leq (s,[x])$ is a branch point, by Lemma \ref{lem 3}.
\end{proof}

In the situation of Lemma \ref{lem 5}, one says that $(s_{0},[x_{0}])$ is the {\it maximal branch point below} $(s,[x])$.

If $(s,[x])$ is a branch point, then the maximal branch point below $(s,[x])$ is $(s,[x])$, by construction.

\begin{lemma}\label{lem 6}
  Suppose that $(s_{0},[x_{0}])$ and $(t_{0},[y_{0}])$ are maximal branch points below the points $(s,[x])$ and $(t,[y])$ in $\Gamma_{k}(X)$, respectively.
  Then $(s_{0},[x_{0}]) \cup (t_{0},[y_{0}])$ is the maximal branch point below $(s,[x]) \cup (t,[y])$.
\end{lemma}

\begin{proof}
  Suppose that $s \leq t$.
  
We have 
\begin{equation*}
(s_{0},[x_{0}]) \cup (t_{0},[y_{0}]) \leq (s,[x]) \cup (t,[y]).
\end{equation*}
and
$(s_{0},[x_{0}]) \cup (t_{0},[y_{0}])$
is a branch point by Lemma \ref{lem 3}.

Write
\begin{equation*}
(v,[z]) = (s_{0},[x_{0}]) \cup (t_{0},[y_{0}]). 
\end{equation*}

\noindent
1)\
Suppose that $v \leq t$.
Then
\begin{equation*}
  (t_{0},[y_{0}]) \leq (t,[y])=(t,[y_{0}])
\end{equation*}
and \begin{equation*}
  (t_{0},[y_{0}]) \leq (v,[z])=(v,[y_{0}]),
\end{equation*}
so that
\begin{equation*}
  (v,[z]) = (v,[y_{0}]) \leq (t,[y_{0}])= (t,[y])
\end{equation*}
since $v \leq t$.

Also, $(s_{0},[x_{0}]) \leq (s,[x])$ and $(s_{0},[x_{0}]) \leq (v,[z]) \leq (t,[y])$ so that $(s,[x]) \leq (t,[y])$.
\medskip

Then $(s_{0},[x_{0}]) \leq (t_{0},[y_{0}])$ by maximality, and it follows that
\begin{equation*}
  (s_{0},[x_{0}]) \cup (t_{0},[y_{0}]) = (t_{0},[y_{0}])
\end{equation*}
is the maximal branch point below
\begin{equation*}
  (s,[x]) \cup (t,[y]) = (t,[y])
\end{equation*}

\noindent
2)\
  Suppose that $v >t$. Then $(s,[x]) = (s,[x_{0}]) \leq (v,[z])$ and $(t,[y]) = (t,[y_{0}]) \leq (v,[z])$ because $s \leq t <v$, so that
\begin{equation*}
  (s,[x]) \cup (t,[y]) \leq (s_{0},[x_{0}]) \cup (t_{0},[y_{0}]),
\end{equation*}
Thus, $(s_{0},[x_{0}]) \cup (t_{0},[y_{0}]) = (s,[x]) \cup (t,[y])$ is a branch point, 
by Lemma \ref{lem 3}.
\end{proof}

\begin{lemma}
The poset inclusion $\alpha: \Br_{k}(X) \to \Gamma_{k}(X)$ has an inverse
\begin{equation*}
  max: \Gamma_{k}(X) \to \Br_{k}(X),
\end{equation*}
up to homotopy, and $\Br_{k}(X)$ is a strong deformation retract of $\Gamma_{k}(X)$.
\end{lemma}

\begin{proof}
Lemma \ref{lem 5} implies that every vertex $(s,[x])$ of $\Gamma_{k}(X)$ has a unique maximal branch point $(s_{0},[x_{0}])$ such that $(s_{0},[x_{0}]) \leq (s,[x])$. Set
\begin{equation*}
  max(s,[x]) = (s_{0},[x_{0}]).
\end{equation*}
The maximality condition implies that $max$ preserves the ordering. The composite
$max \cdot \alpha$ is the identity on $\Br_{k}(X)$, and the relations $(s_{0},[x_{0}]) \leq (s,x)$ define a homotopy $max \cdot \alpha \leq 1$ that restricts to the identity on $\Br_{k}(X)$.
\end{proof}

Return to the inclusion $i: X \subset Y \subset \mathbb{R}^{n}$ of finite data sets. Suppose that $d_{H}(X^{k+1}_{dis},Y^{k+1}_{dis}) < r$ and that $L_{s,k}(Y)$ is non-empty, as in the statement of Theorem \ref{th 1}.

Write $i_{\ast}: \Br_{k}(X) \to \Br_{k}(Y)$ for the composite poset morphism
\begin{equation*}
  \Br_{k}(X) \xrightarrow{\alpha} \Gamma_{k}(X) \xrightarrow{i_{\ast}} \Gamma_{k}(Y) \xrightarrow{max} \Br_{k}(Y)
  \end{equation*}
This map takes a branch point $(s,[x])$ to the maximal branch point below $(s,[i(x)])$. 
\medskip

\noindent
{\bf Remark}:\
The map $i_{\ast}: \Br_{k}(X) \to \Br_{k}(Y)$ only preserves least upper bounds up to homotopy. Suppose that $(s,[x])$ and $(t,[y])$ are branch points of $X$, and let $(s_{0},[x_{0}]) \leq (s,[i(x)])$ and $(t_{0},[y_{0}]) \leq (t,[i(y)])$ be maximal branch points below the images of $(s,[x])$ and $(t,[y])$ in $\Gamma_{k}(Y)$. Then $(s_{0},[x_{0}]) \cup (t_{0},[y_{0}])$ is the maximal branch point below $(s,[i(x)]) \cup (t,[i(y)])$ by Lemma \ref{lem 6}, but it may not be the maximal branch point below $i_{\ast}((s,[x]) \cup (t,[y]))$.
\medskip

Poset morphisms $\theta_{\ast}: \Br_{k}(Y) \to \Br_{k}(X)$ and $\sigma_{\ast}: \Br_{k}(X) \to \Br_{k}(X)$ are similarly defined, by the poset morphism $\theta: \Gamma_{k}(Y) \to \Gamma_{k}(X)$ given by $(t,[y]) \mapsto (t+2r,[\theta(y)])$, and the shift morphism $\sigma: \Gamma_{k}(X) \to \Gamma_{k}(X)$ given by $(s,[x]) \mapsto (s+2r,[x])$. These maps again preserve least upper bounds up to homotopy.
\medskip

\noindent
1)\ Consider the poset maps
\begin{equation*}
  \Br_{k}(X) \xrightarrow{i_{\ast}} \Br_{k}(Y) \xrightarrow{\theta_{\ast}} \Br_{k}(X).
\end{equation*}

If $(s,[x])$ is a branch point for $X$, choose maximal branch points $(s_{0},[x_{0}]) \leq (s,[i(x)]$ for $Y$, $(s_{1},[x_{1}]) \leq (s_{0}+2r,[\theta(x_{0})])$ and $(v,[y]) \leq (s+2r,[x])$ below the respective objects.

Then $\theta_{\ast}i_{\ast}(s,[x]) = (s_{1},[x_{1}])$, and there is a natural relation
\begin{equation*}
  \theta_{\ast}i_{\ast}(s,[x]) = (s_{1},[x_{1}]) \leq (v,[y]) = \sigma_{\ast}(s,[x]) 
\end{equation*}
by a maximality argument.
We therefore have a homotopy of poset maps
\begin{equation}\label{eq 4}
  \theta_{\ast}i_{\ast} \leq \sigma_{\ast}: \Br_{k}(X) \to \Br_{k}(X).
\end{equation}

\noindent
2)\ Similarly, if $(t,[y])$ is a branch point of $Y$, then
\begin{equation*}
  i_{\ast}\theta_{\ast}(t,[y]) \leq \sigma_{\ast}(t,[y]),
\end{equation*}
giving a homotopy
\begin{equation}\label{eq 5}
  i_{\ast}\theta_{\ast} \leq \sigma_{\ast}: \Br_{k}(Y) \to \Br_{k}(Y).
  \end{equation}

The construction of the poset maps $i_{\ast}$, $\theta_{\ast}$ and $\sigma_{\ast}$, together with the relations (\ref{eq 4}) and (\ref{eq 5}), complete the proof of Theorem \ref{th 2}.
\medskip

There are relations
\begin{equation}
  (s,[x]) \leq \sigma_{\ast}(s,[x]) \leq (s+2r,[x])
\end{equation}
for branch points $(s,[x])$.
It follows that the poset map $\sigma_{\ast}: \Br_{k}(X) \to \Br_{k}(X)$ is homotopic to the identity on $\Br_{k}(X)$.

It also follows that $\sigma_{\ast}(s,[x]) = (t,[x])$ is close to $(s,[x])$ in the sense that $t-s \leq 2r$. Thus, the branch points $(s,[x])$ and $\theta_{\ast}i_{\ast}(s,[x])$ have a common upper bound, namely $\sigma_{\ast}(s,[x])$, which is close to $(s,[x])$.

The subobject of $\Br_{k}(X)$ consisting of all branch points of the form $(s,[x])$ as $s$ varies has an obvious notion of distance: the distance between points $(s,[x])$ and $(t,[x])$ is $\vert t-s \vert$.  

If $(t,[y])$ is a branch point of $\Gamma_{k}(Y)$, the branch point $\sigma_{\ast}(t,[y])$ is similarly an upper bound for $(t,[y])$ and $i_{\ast}\theta_{\ast}(t,[y])$ that is close to $(t,[y])$.

\bibliographystyle{plain}
\bibliography{spt}

\end{document}